\documentclass[a4paper,10pt]{article}
\usepackage{amsmath, amsthm, amssymb}
\RequirePackage[colorlinks]{hyperref}
\RequirePackage{natbib}
\numberwithin{equation}{section}

\newtheorem{definition}{Definition}
\newtheorem{assumption}{Assumption}
\newtheorem{proposition}{Proposition}
\newtheorem{theorem}{Theorem}

\newcommand{\mR}{{\mathcal R}}
\newcommand{\bN}{{\mathbb N}}
\newcommand{\bR}{{\mathbb R}}
\newcommand{\bZ}{{\mathbb Z}}
\newcommand{\bP}{{\mathbb P}}
\newcommand{\bE}{{\mathbb E}}

\newcommand{\mC}{{\mathcal C}}

\begin{document}

\title{Insensitive, maximum stable allocations converge to proportional fairness}
\author{N.S. Walton\footnote{ Statistical Laboratory, 
Centre for Mathematical Sciences,
University of Cambridge,
Cambridge CB3 0WB,
United Kingdom 
n.s.walton@statslab.cam.ac.uk}}

\maketitle

\begin{abstract}
We describe a queueing model where service is allocated as a function of queue sizes. We consider allocations policies that are insensitive to service requirements and have a maximal stability region. We take a limit where the queueing model become congested. We study how service is allocated under this limit. We demonstrates that the only possible limit allocation is one that maximizes a proportionally fair optimization problem.

\end{abstract}

\section{Introduction}
Consider a communication network. Documents arrive and are transferred across the different routes of the network. In a communication network, each document transfer receives a rate which may vary through time, depending on the number of transfers present on each route. What sort of behaviour might we want such a network to have?

Perhaps, we would like our communication network to be stable, to be able to cope with the rate that work arrives? If a communication network does not allocate its capacity well then instability can arise. So we want to consider allocation policies that avoid instability whenever it is possible. We call such policies \textit{maximum stable}.

We may, also, like our communication network to be \textit{insensitive} to different document size distributions? Famous examples, such as the Erlang B formula and the processor sharing queue, are know to be insensitive to different job size distributions. Like with the Erlang B formula, an advantage of insensitivity is that stationary statistics can be described without an explicit knowledge all traffic parameters. In addition, like with the processor sharing queue, an advantage of insensitivity is that the expected service time of a document is proportional to its size. Thus, each unit of a document's work can be expected to be treated equally.

A further advantage of a processor sharing discapline is that it allocates capacity proportional to the number of customers in each class. In this way, as the queue gets congested, each customer class is served at a consistent rate. Perhaps, we might want to extend this property to our communication network? In our communication network, as the number of documents on each route gets large, we might want to allocate service relative to number of document transfers on each route. This will result in a \textit{limiting allocation policy} that assigns service according to the fraction of documents on each route.

We have described three desirable properties for our communication network to satisfy: to be maximum stable, to be insensitive and to have a limiting allocation policy. In this paper, we reason that the only limiting allocation policy that can arise from a maximum stable, insensitive network is the policy that solves the optimization problem: 
\begin{equation}\label{pf}
\text{maximize}\quad \sum_{r\in\mR} n_r\log \Lambda_r\quad \text{over} \quad \Lambda\in\mC.
\end{equation}
In this optimization problem, $n_r$ gives the proportion of documents on route $r\in\mR$, $\mC$ gives the set of feasible allocations and $\Lambda=(\Lambda_r:r\in\mR)$ gives a feasible allocation. The allocation policy $\Lambda^{PF}(n)$ that solves the optimization problem \eqref{pf} for each $n=(n_r:r\in\mR)$ is called the \textit{proportionally fair} policy.

We note that by assuming these three initial properties, we have gained an addition property: our network is attempting to \textit{optimize} the service rate allocated to each route. We did not initially prescribe this optimization structure, instead this behaviour is implicitly implied by our networking assumptions.
\subsection{Literature review}
The most important reference for this paper is \cite{Ma07}. Massouli\'e observes that two insensitive allocations: balanced fairness and modified proportional fairness, have a large deviations rate function given by the optimization problem \eqref{pf}. Remark 2 of \cite{Ma07} notes that if balanced fairness has a limiting allocation policy then is must be proportionally fair. Massouli\'e proves modified proportional fairness is maximum stable, insensitive and  that proportional fairness is its limiting allocation policy.

 \cite{Wa09} proves another allocation policy arising from multi-class queueing networks is maximum stable, insensitive and has proportional fairness as its limiting allocation policy. This insensitive allocation policy was first published by \cite{BoPr04}.

Early work on insensitivity our modelling context is due to \cite{Wh85}. A good paper for results, properties and historical review on the insensitivity of allocation policies is \cite{BoPr02}. Procedures for calculating stationary statistics of insensitive networks are discussed in \cite{BPRV03, MaJo10}. Discussions on the expected service times of insensitive allocations are given in \cite{BoPr02} and \cite{KMW09}.

\subsection{Paper Organization}

In Section \ref{sec2}, we present the main allocation policies of interest: the proportionally fair policy and the insensitive policies. In addition, we define what it means for an allocation policy to have a limiting allocation policy. In Section \ref{sec3}, we define our stochastic model of a communication network. We call this model a stochastic flow level model. We, also, define what we mean by maximum stable.
In Section \ref{sec4}, we define insensitivity for a stochastic flow level model, and we state a known insensitivity result. In Section \ref{sec5}, we prove the main result of this paper that the only limit allocation policy of a maximum stable, insensitive policy is the proportionally fair policy, see Theorem \ref{conj 1}. In Section \ref{sec6}, we ask if all maximum stable, insensitive allocations converge to proportional fairness without assuming that a limiting allocation policy exists? We show this is not possible with a counter example: Proposition \ref{prop 1}.

\section{Allocation Policies}\label{sec2}
An \textit{allocation policy} allocates the rate of document transfer across a network. Such policies allocate given the number of documents in transfer on the routes of the network and subject to the capacity constraints of the network. More explicitly, we let the set $\mR$ index the set of routes of the network, and we let the components of the vector $n=(n_r:r\in\mR)\in\bZ_+^\mR$ give the number of documents in transfer on each route. An allocation policy is a function  $\Lambda: \bZ_+^\mR \rightarrow \mC$ where $\mC\subset \bR_+^\mR$ is a closed, bounded, convex set with a non-empty interior. The set $\mC$ is the set of allocations that can be scheduled. We call $\mC$ the \textit{schedulable region}.

An allocation policy of central interest to this paper is defined as follows.

\begin{definition}[Proportionally fair policy]
The allocation policy
$\Lambda^{PF}(n) = (\Lambda^{PF}_r(n): r \in \mR)\in\mC$, defined for $n\in \bZ_+^\mR$,
is called \it
proportionally fair \rm if for all $n\in\bR_+^{\mR}$\vspace{-0.3cm}
\begin{equation*}
 n_r=0\quad \text{implies} \quad \Lambda^{PF}_r(n)=0,
\end{equation*}
and if $\Lambda^{PF}(n)$ solves the optimization problem\footnote{We assume throughout
this paper that $xlog x =0$ for $x=0$, and we adopt the convention  that
$0^0=1$. }
\begin{equation}
\text{maximise}\qquad \sum_{r\in\mR} n_r\log \Lambda_r \qquad \text{over}\qquad \Lambda\in \mC. \label{pf}
\end{equation}
\end{definition}

Proportional fairness was first defined as policy for sharing bandwidth by \cite{Ke97}.

We will shortly introduce a stochastic model of document transfer and allocation. 
For this model, we will be interested in allocation policies that show insensitivity to document sizes. Before we define this model and this insensitivity property, we will define the set of allocation policies which will exhibit insensitivity.

\begin{definition}[Insensitive allocation policies]\label{insens 1}
An allocation policy $\Lambda(\cdot)$ is \textit{insensitive} if there exists a function $\Phi:\bZ^{\mR} \rightarrow \bR_+$ such that $\Phi(0)=1$, $\Phi(n)=0$ $\forall n\notin \bZ_+^{\mR}$ and 
\begin{equation}
 \Lambda_r(n)=\frac{\Phi(n-e_r)}{\Phi(n)},\qquad \forall n\in\bZ_+^{\mR}.\label{insens lambda}
\end{equation}
Here and subsequently, $e_r$ will denote the $r$th unit vector in $\bR_+^{\mR}$.
\end{definition}

We want our allocation policy to have a limit when we allow the routes of the network to become congested. For this reason, we consider limiting allocation policies.

\begin{definition}[Limiting allocation policy]
An allocation policy $\Lambda(\cdot)$ is a \textit{limiting allocation policy} if there exists $\tilde{\Lambda}(\cdot)$ such that for all $n\in\bR_+^\mR$,
\begin{equation}
\Lambda(cn + \text{\Large{o}}(1)) \xrightarrow[c\rightarrow\infty]{} \tilde{\Lambda}(n). \label{limiting}
\end{equation}
Given the conventional `Big-O notation', statement \eqref{limiting} is a shorthand for the convergence statement: for all bounded sequences $\{d_c\}_{c}$ such that $cn+d_c\in\bZ_+^\mR$ 
\begin{equation*}
 \Lambda_r(cn+d_c) \xrightarrow[c\rightarrow\infty]{} \tilde{\Lambda}_r(n), \qquad r\in\mR.
\end{equation*}
\end{definition}

\section{A Stochastic Model}\label{sec3}
We, now, define a stochastic model that allocates service with an allocation policy. This model can be thought of as a model of document transfer across the Internet. We introduce our stochastic model for an allocation policy $\Lambda(\cdot)$.

Documents to be transfered arrive as a Poisson process. Route $r\in\mR$ documents arrive as an independent Poisson process of rate $\nu_r>0$. Each document has a size that is divided into stages. Each stage has a size that is independent exponentially distributed with mean $\delta$. Each route $r\in\mR$ document has a number of stages that is independent and equal in distribution to a random variable $L_r$. We assume $L_r$ has a finite mean and we define~$\mu_r^{-1}:=\delta \bE L_r$,~the mean size of a route $r$ document. When there are $n=(n_r: r\in\mR)$ documents in transfer on each route, each route $r$ documents is served at rate $\frac{\Lambda_r(n)}{n_r}$. Documents are then processed at this rate until the number of documents in transfer changes either by a document departure or arrival. 

To be explicit, given there are $n_{rs}$ route $r$ documents with $s$ stages remaining, the state of this Markov chain is $x=(n_{rs}: r\in\mR, s\in\bN)$. Letting $e_{rs}$ be the $rs$th unit vector, the non-zero transition rates of this Markov chain are
\begin{align*}
 q(x,x+e_{rs}) &= v_r\bP(L_r=s),\\
 q(x,x-e_{r1}) &= \delta  \Lambda_r(n) \frac{n_{r1}}{n_r} \qquad \text{if } n_{r1}>0, \\
 q(x,x-e_{rs}+e_{rs-1}) &= \delta \Lambda_r(n) \frac{n_{rs}}{n_r} \qquad \text{if } n_{rs}>0\text{ and } s\geq 1.
\end{align*}
We call this Markov chain a \textit{stochastic flow level model operating under allocation} $\Lambda(\cdot)$. This model was introduced by \cite{MaRo98}. 

We define the vector of \textit{traffic intensities} $\rho\in \bR_+^{\mR}$ by $\rho_r:=\frac{\nu_r}{\mu_r}$, $r\in\mR$. Also, for the schedulable region $\mC$, we let $\mC^{\circ}$ be the interior of $\mC$.

We consider stochastic flow level models that are stable when possible. For this purpose, we will later use the following assumption.

\begin{assumption}[Maximum Stable]\label{Assump 1}
For $\Lambda(\cdot)$, an allocation policy on schedulable region $\mC$,  we say that the stochastic flow level model operating under $\Lambda(\cdot)$ is maximum stable if it is positive recurrent for all
\begin{equation}
 \rho\in \mC^{\circ}.
\end{equation}
\end{assumption}
A stochastic flow level model is transient for all $\rho$ outside the region $\mC$. Also, there exist allocation policies that are positive recurrent for all $\rho\in\mC^{\circ}$. Notably, \cite{Ma07} proves that the proportionally fair policy is positive recurrent for all $\rho\in\mC^{\circ}$. Thus, in words, Assumption \ref{Assump 1} states that a stochastic flow level model under $\Lambda(\cdot)$ is stable in the largest possible region.





\section{Insensitivity}\label{sec4}
We are interested in stochastic flow level models that show robustness to the distribution of document sizes. For this reason, we consider insensitive stochastic flow level models.

Take a stochastic flow level model operating under allocation $\Lambda(\cdot)$. For such a model, let $\pi=(\pi(n):n\in\bZ_+^{\mR})$ be the stationary distribution of the number of documents in transfer on each route. In general, $\pi$ is a function of the parameters of our model: $(\nu_r:r\in\mR)$, $(L_r: r\in\mR)$ and $\delta$.

\begin{definition}[Insensitive Stochastic Flow Level Model]\label{insens 2}
We define a stochastic flow level model to be \textit{insensitive} if $\pi(n)$ can be expressed as a function parameters $(\nu_r:r\in\mR)$ and $(\mu_r:r\in\mR)$. 
\end{definition}

In words, this definition says that we do not need to know the precise distribution of document sizes to know the stationary behaviour of our model; instead, we only need to know the mean size of the documents. 

Derived from the work of \cite{Wh85}, a key observation made by \cite{BoPr02} was that Definitions \ref{insens 1} and \ref{insens 2} are equivalent in the following sense.

\begin{proposition} \label{insense prop} 
a) A stochastic flow level model is insensitive iff it is operating under an allocation that is insensitive.\\ 
\noindent b) For an insensitive stochastic flow level model,
\begin{equation}
\pi(n)=\frac{\Phi(n)}{B(\rho)}\prod_{r\in\mR} \rho_r^{n_r},\qquad\text{for}\quad n\in\bZ_+^{\mR},\; \rho\in\mC^{\circ}, \label{ed}
\end{equation}
where
\begin{equation}\label{scaling const}
 B(\rho):=\sum_{n\in\bZ_+^{\mR}} \Phi(n)\prod_{r\in\mR} \rho_r^{n_r}.
\end{equation}
\noindent c) When document sizes are exponentially distributed, an insensitive stochastic flow level model is reversible.
\end{proposition}

The proof of this proposition involves checking that reversibility is equivalent to insensitivity for these queueing networks. The stationary distribution (\ref{ed}) then arises from the detail balance equations. See \cite{BoPr02} for a proof.

\section{The main result}\label{sec5}
We can now prove the main result of this paper.

\begin{theorem}\label{conj 1}
For $\Lambda(\cdot)$ an insensitive, maximum stable, limiting allocation policy
\begin{equation}
 \Lambda(cn + \text{\Large{o}}(1)) \xrightarrow[c\rightarrow\infty]{} \Lambda^{PF}(n),\qquad\quad \text{ for all }n\in\bR_+^\mR.\label{limit thrm}
\end{equation}
Moreover, for $\rho\in\mC^\circ$
\begin{equation}
 \lim_{c\rightarrow\infty} \frac{1}{c} \log \pi(\lfloor cn\rfloor)= -\max\;\; \sum_{r\in\mR} n_r \log \frac{\Lambda_r}{\rho_r} \;\; \mathrm{over}\;\; \Lambda \in\mC. \label{ldp}
\end{equation}
Here and hereafter, we use the convention that $\lfloor x\rfloor := (\lfloor x_r\rfloor : r\in\mR)$ is the lower integer part of each component of $x\in\bR_+^\mR$. 
\end{theorem}
\begin{proof}
To prove \eqref{ldp}, we will construct an upper bound and a lower bound. To prove the upper bound, we apply a Chernoff bound. To prove the lower bound, we bound our reversible process along a sequence of transitions from $\lfloor \epsilon cn \rfloor$ to $\lfloor cn \rfloor$. The result \eqref{limit thrm} is a consequence of \eqref{ldp}. 

We start with the upper bound: by a Chernoff bound and Proposition \ref{insense prop}b), for $\theta\in\bR^{\mR}$
\begin{equation*}
\pi(\lfloor cn \rfloor) \leq \bE e^{\sum_r \theta_r(N_r-\lfloor cn_r\rfloor)}=
\begin{cases}
 \frac{B(\rho e^\theta)}{B(\rho)}e^{-\sum_r \theta_r\lfloor cn_r\rfloor} & \text{if } \rho e^{\theta} \in \mC^{\circ},\\
\infty & \text{if } \rho e^{\theta} \notin \mC.
\end{cases}
\end{equation*}
Here we use the shorthand $\rho e^{\theta} := (\rho_re^{\theta_r} : r\in\mR)$. Thus, for $\rho e^{\theta} \in \mC^{\circ}$
\begin{equation}
\limsup_{c\rightarrow \infty} \frac{1}{c} \log \pi(\lfloor cn \rfloor) \notag
\leq \lim_{c\rightarrow\infty} \frac{1}{c} \log \bE e^{\sum_r \theta_r(N_r-\lfloor cn_r\rfloor)}=
 -\sum_{r\in\mR} n_r\theta_r. \label{upper bound}
\end{equation}
Minimizing over $\theta \in\bR^{\mR}$ with $\rho e^{\theta} \in \mC^{\circ}$ gives the required upperbound:
\begin{align*}
\limsup_{c\rightarrow \infty} \frac{1}{c} \log \pi(\lfloor cn \rfloor) &\leq -\max \; \sum_{r\in\mR} n_r\theta_r \;\; \text{over} \;\; \rho e^{\theta} \in\mC \\
&= -\max \;\ \sum_{r\in\mR} n_r\log \frac{\Lambda_r}{\rho_r} \;\; \text{over} \;\; \Lambda \in\mC.
\end{align*}
Above, we make the substitution $\Lambda_r=\rho_r e^{\theta_r}$ for each $r\in\mR$.


Now we prove the lower bound. We assume, for simplicity, that document sizes are exponentially distributed. Thus, by Proposition \ref{insense prop}c) our Markov process is reversible.  We want to consider a sequence of transitions of this processes close to the line $tn$, $t>0$. For $t> 0$, let $m^0,m^1,m^2...$ be the distinct values of $\lfloor tn\rfloor$ as $t$ increases to infinity. This sequence is increasing but a transition from state $m^k$ to $m^{k+1}$ might take more than one document arrival. If so, split the transition from $m^k$ to $m^{k+1}$ into increasing transitions each involving a single arrival. Let $n^0, n^1, n^2...$ be the resulting sequence. Note that the sequence $n^0, n^1, n^2,...$ is an increasing sequence of transitions and close $tn$. This is because for all $k\geq 0$, $n^{k+1}=n^{k}+e_{r_k}$ for some route $r_k$ and because each point $n^k$ is within distance $|\mR|$ of the line $tn$, $t>0$. \footnote{In this paper, we judge distances by the $L^1$ norm $|n|:=\sum_r |n_r|$.}

Define $k(t)$ to be the term in our sequence which equals $\lfloor tn \rfloor$ and also take $\epsilon>0$. By the reversibility of our process
\begin{equation}\label{pie}
 \pi(\lfloor cn\rfloor ) = \left[ \prod_{k=k(c\epsilon)+1}^{k(c)} \frac{\rho_{r_k}}{\Lambda_{r_k}(n^k)}\right] \times \pi(\lfloor cn\epsilon \rfloor).
\end{equation}
In expression \eqref{pie}, by taking $c\epsilon$ suitably large, we can make all the $\Lambda(n^k)$ terms close to $\tilde{\Lambda}(n)$. In particular, by our limit allocation assumption, for all $\delta>0$ there exists $c$ such that for all $k>k(c\epsilon)$
\begin{equation}\label{lambda til}
 \Lambda_{r}(n^k) e^{-\delta} \leq \tilde{\Lambda}_r(n) \leq \Lambda_{r}(n^k) e^{\delta},\qquad \forall r\in\mR\text{ with }n_r>0.\footnote{We note that $\tilde{\Lambda}_r(n)>0$ if $n_r>0$ as otherwise $\sum_k \pi(n^k)$ would diverge.}
\end{equation}
Noting that exactly $\lfloor cn_r\rfloor-\lfloor cn_r\epsilon\rfloor$ route $r$ transitions are used in \eqref{pie}, and applying \eqref{lambda til} to \eqref{pie} gives that
\begin{align*}
 &\liminf_{c\rightarrow\infty} \frac{1}{c} \log \pi(\lfloor cn\rfloor) \\
 &\geq \liminf_{c\rightarrow\infty} - \sum_{r\in\mR}\!\!\!\sum_{\substack{ k:\; r_k=r \\ k(c\epsilon)<k\leq k(c) }}\!\!\! \frac{1}{c} \log \Big(\frac{\Lambda_{r}(n^k)}{\rho_{r}}\Big) +   \liminf_{c\rightarrow\infty} \frac{1}{c} \log \pi(\lfloor cn\epsilon \rfloor)\\
&\geq - \sum_{r\in\mR} (1-\epsilon)n_r \log \Big(\frac{\tilde{\Lambda}_r(n)e^{\delta}}{\rho_r}\Big) + \epsilon\liminf_{c\rightarrow\infty} \frac{1}{c} \log \pi(\lfloor cn \rfloor)
\end{align*}
Letting $\delta\rightarrow 0$, cancelling terms and dividing by $1-\epsilon$ gives
\begin{equation}
 \liminf_{c\rightarrow\infty} \frac{1}{c} \log \pi(\lfloor cn \rfloor) \geq -\sum_{r\in\mR} n_r \log \frac{\tilde{\Lambda}_r(n)}{\rho_r}. \label{lower bound 2}
\end{equation}
Combining our upperbound \eqref{upper bound} with our new lower bound \eqref{lower bound 2} implies
\begin{equation*}
 \sum_{r\in\mR} n_r \log \frac{\tilde{\Lambda}_r(n)}{\rho_r}\geq \max \;\ \sum_r n_r\log \frac{\Lambda_r}{\rho_r} \;\; \text{over} \;\; \Lambda \in\mC.
\end{equation*}
But $\tilde{\Lambda}(n)\in\mC$, as  it is the limit if elements of $\mC$. Thus $\tilde{\Lambda}_r(n)=\Lambda^{PF}_r(n)$. This completes part b) of the proof and, given of our upper bound \eqref{upper bound} and lower bound \eqref{lower bound 2}, we have also proven part a).
\end{proof}
%
%
\section{A counter example}\label{sec6}
Theorem \ref{conj 1} required the assumption that a limiting allocation policy existed. Can this assumption be completely removed? The answer to this is no, as the following proposition demonstrates.
\begin{proposition}\label{prop 1}
 From any insensitive, maximum stable, limiting allocation policy $\Lambda(\cdot)$, we can construct $\hat{\Lambda}(\cdot)$, an insensitive, maximum stable allocation policy that does not have a limit and, thus, does not converge to a proportionally fair policy.
\end{proposition}
\begin{proof}
Let $\Phi(n)$ be the potential function associated with $\Lambda(n)$ in \eqref{insens lambda}. By Theorem \ref{conj 1}, $\Lambda(n)$ limits to a proportionally fair allocation. Define
\begin{equation*}
 \hat{\Phi}(n):= \alpha^k \Phi(n),\qquad \text{for the } k\in\bN \text{ such that } 2^k\leq |n| < 2^{k+1}.
\end{equation*}
We assume $\alpha>1$ and we take $|n|:=\sum_r |n_r|$. For each $r\in\mR$, observe that
\begin{equation*}
 \hat{\Lambda}_r(n)=
\begin{cases}
 \Lambda_r(n), & \text{if }\; 2^k< |n| < 2^{k+1} \text{ for some }k\in\bN,\\
\alpha^{-1}\Lambda_r(n), & \text{if }\; |n| = 2^{k} \text{ for some }k\in\bN.
\end{cases}
\end{equation*}
Since $\alpha>1$ and $0\in\mC$, $\hat{\Lambda}(n)\in\mC$ for all $n\in\bZ_+^\mR$. We can let the sequence $cn$ hit $|cn|=2^k$ for all $k$ suitably large. For example, we can do this by letting $|n|=2^m$ for some $m\in\bN$. In this case, for $c_k=2^k$
\begin{equation*}
 \tilde{\Lambda}(c_kn)= \frac{\Lambda(c_kn)}{\alpha} \xrightarrow[k\rightarrow\infty]{}  \frac{\Lambda^{PF}(n)}{\alpha}\neq \Lambda^{PF}(n).
\end{equation*}
So along this subsequence, the allocation policy $\tilde{\Lambda}(c_kn)$ does not limit to a proportionally fair allocation. Along another subsequence, e.g. $c'_k=2^k-1$, the allocation policy $\tilde{\Lambda}(c'_kn)$ would limit to a proportionally fair allocation. Thus no limit allocation policy exists for $\tilde{\Lambda}(\cdot)$.

It remains to show that $\tilde{\Lambda}(\cdot)$ is a maximum stable allocation. Take $\rho\in\mC^{\circ}$, there exists $\epsilon>0$ such that $\rho e^\epsilon=(\rho_r e^\epsilon: r\in\mR) \in\mC^{\circ}$. Note also there exists $N$ such that $\forall |n|>N$, $\alpha^{\log_2 |n|} < e^{\epsilon |n|}$. Let $\hat{B}(\rho)$ and $B(\rho)$ be the normalizing constants defined by \eqref{scaling const} for $\hat{\Phi}(\cdot)$ and $\Phi(\cdot)$, respectively. We know that $B(\rho e^\epsilon)<\infty$. We now show that $\hat{B}(\rho)<\infty$.
\begin{align*}
 \hat{B}(\rho)&=\sum_{k=0}^\infty \sum_{n: 2^k < |n| \leq 2^{k+1}} \alpha^k \Phi(n) \prod_{r\in\mR} \rho_r^{n_r}\\
&\leq \sum_{n: |n|>N } e^{\epsilon |n|}\Phi(n) \prod_{r\in\mR} \rho_r^{n_r}+ \sum_{n: |n|\leq N } \alpha^{\log_2 |n|}  \Phi(n) \prod_{r\in\mR} \rho_r^{n_r} \\
&=B(\rho e^\epsilon) + \sum_{n: |n|\leq N } \alpha^{\log_2 |n|}  \Phi(n) \prod_{r\in\mR} \rho_r^{n_r}\\
&< \infty.
\end{align*}
Thus, as the normalizing constant $\hat{B}(\rho)$ is finite for all $\rho\in\mC^\circ$, our allocation policy $\hat{\Lambda}(\cdot)$ is maximum stable.

\end{proof}
%
%
%
%
%

\bibliographystyle{spbasic}
\bibliography{../references/references} 

\end{document}